\documentclass{amsart}

\usepackage{amsfonts,amsmath,amssymb,amstext,amsthm}
\usepackage{graphicx,color}
\usepackage{mathrsfs}
\usepackage{pdfpages}
\usepackage{tabularx}

\usepackage{geometry}
\usepackage{marginnote}
\usepackage{appendix}

\usepackage{calligra}

\usepackage{color}
\usepackage[normalem]{ulem}

\usepackage{array}
\usepackage{makecell}
\usepackage{tabularx}
\usepackage[framemethod=tikz]{mdframed}

\newcommand{\nwc}{\newcommand}

\newcommand{\com}{\mathbb{C}}

\newcommand{\cal}{\mathcal}
\newcommand{\op}[1]{{\mathcal O}_{\mathbb{P}^{#1}}}
\newcommand{\p}[1]{{\mathbb{P}^{#1}}}
\newcommand{\opn}{{\mathcal O}_{\mathbb{P}^{n}}}
\newcommand{\pn}{{\mathbb{P}^{n}}}
\newcommand{\tp}[1]{{\rm T}{\mathbb{P}^{#1}}}
\newcommand{\tpn}{{\rm T}{\mathbb{P}^{n}}}
\newcommand{\iz}{\mathcal{I}_{Z}}

\nwc{\aaa}{\sF }
\nwc{\aab}{\bar{\mathfrak{a}}}
\nwc{\aal}{\sF '}
\nwc{\aap}{\sF _{P}}

\nwc{\bbb}{\mathfrak{b}}
\nwc{\bbp}{\mathfrak{b}_{P}}

\nwc{\C}{\mathbb{C}}
\nwc{\cb}{\overline{C}}
\nwc{\ccc}{\mathcal{C}}
\nwc{\ch}{\widehat{C}}
\nwc{\cin}{\textbf{(v)}}
\nwc{\cl}{C'}
\nwc{\cp}{\mathcal{C}_{P}}
\nwc{\cpll}{\mathcal{C}_{P'}}
\nwc{\ct}{\tilde{C}}

\nwc{\dd}{\mathcal{D}}
\nwc{\ddd}{\mathfrak{d}}
\nwc{\ddl}{\mathcal{L}'}
\nwc{\dlp}{\delta_{P}}
\nwc{\doi}{\textbf{(ii)}}

\nwc{\ff}{\mathscr{F}}

\nwc{\G}{{\cal G}}
\nwc{\gon}{{\rm gon}}
\nwc{\gtl}{\tilde{g}}
\nwc{\gud}{g^{1}_{2}}
\nwc{\gtu}{g^{1}_{3}}

\nwc{\hhza}{H^{0}(C,\mathfrak{a})}
\nwc{\hua}{h^{1}(C,\mathfrak{a})}
\nwc{\hza}{h^{0}(C,\mathfrak{a})}

\nwc{\kk}{{\rm K}}

\nwc{\lbd}{\lambda}
\nwc{\lif}{L_{\infty}}

\nwc{\mm}{\mathfrak{m}}
\nwc{\mmp}{\mathfrak{m}_{P}}
\nwc{\mpd}{{\mathfrak{m}_{P}}^{2}}

\nwc{\N}{I\!\!N}
\nwc{\nn}{\mathbb{N}}

\nwc{\obp}{\overline{\mathcal{O}}_P}
\nwc{\ocbux}{\oo _{\bar{C}}\langle 1,x\rangle}
\nwc{\oclux}{\oo _{C'}\langle 1,x\rangle}
\nwc{\ocux}{\oo _{C}\langle 1,x\rangle}
\nwc{\ol}{\mathcal{O}'}
\nwc{\oma}{\Omega (\mathfrak{a})}
\nwc{\omo}{\Omega (\mathcal{O})}
\nwc{\oo}{\mathcal{O}}
\nwc{\ooh}{\widehat{\mathcal{O}}}
\nwc{\opc}{\mathcal{O}_{P,C}}
\nwc{\oph}{\widehat{\mathcal{O}}_{P}}
\nwc{\opl}{\mathcal{O}_{P}'}
\nwc{\oplc}{\mathcal{O}_{P,C}'}
\nwc{\opll}{\mathcal{O}_{P'}}
\nwc{\opt}{\tilde{\mathcal{O}}_{P}}
\nwc{\optt}{{\mathcal{O}}_{\tilde{P}}}
\nwc{\oq}{\mathcal{O}_{Q}}
\nwc{\oqt}{\tilde{\mathcal{O}}_{Q}}
\nwc{\ot}{\tilde{\mathcal{O}}}
\nwc{\overop}{\bar{\oo}_{P}}

\nwc{\pb}{\overline{P}}
\nwc{\pgmd}{\mathbb{P}^{g+2}}
\nwc{\pgmu}{\mathbb{P}^{g+1}}
\nwc{\pp}{\mathbb{P}}
\nwc{\prv}{\noindent\textbf{Proof}:}
\nwc{\pt}{\tilde{P}}
\nwc{\ptl}{\tilde{P}}
\nwc{\pum}{\mathbb{P}^{1}}
\nwc{\carta}{\mathfrak{U}}

\nwc{\Q}{\;\mbox{{\sf I}}\!\!\!Q}
\nwc{\qb}{\overline{Q}}
\nwc{\qtl}{\tilde{Q}}
\nwc{\qua}{\textbf{(iv)}}

\nwc{\R}{I\!\!R}

\nwc{\sep}{\beq\ast\ \ast\ \ast\enq}
\nwc{\spl}{{S_{P}}'}
\nwc{\spll}{S_{P'}}
\nwc{\ssp}{{\rm S}_{P}}
\nwc{\sss}{{\rm S}}
\nwc{\sys}{\mathcal{L}}

\nwc{\tre}{\textbf{(iii)}}

\nwc{\um}{\textbf{(i)}}

\nwc{\vlp}{\mathcal{V}_{\lambda,P}}
\nwc{\vpt}{v_{\ptl}}
\nwc{\vv}{\mathcal{W}}
\nwc{\vvp}{\mathcal{W}_{P}}
\nwc{\vpb}{v_{\overline{P}}}
\nwc{\vtxp}{\widetilde{V}_{x,P}}
\nwc{\vxp}{V_{x,P}}
\nwc{\vzp}{V_{z,P}}

\nwc{\wol}{\ww\cdot\mathcal{O}'}
\nwc{\wpn}{{\omega _{P}}^{n}}
\nwc{\wwt}{\widetilde{\omega}}
\nwc{\wwtp}{\widetilde{\omega}_P}
\nwc{\ww}{\omega}
\nwc{\wwp}{\omega _{P}}

\nwc{\Z}{{Z\!\!\!Z}}
\nwc{\zz}{\mathbb{Z}}
\newcommand{\Pic}{\operatorname{Pic}}

\def\sF{{\mathscr{F}}}
\DeclareMathOperator{\sing}{Sing}

\newcommand{\dist}{{\mathcal D}{\it ist}}
\newcommand{\sch}{\mathfrak{Sch}_{/\kappa}}
\newcommand{\sets}{\mathfrak{Sets}}
\DeclareMathOperator{\coker}{{coker}}
\DeclareMathOperator{\img}{{im}}

\DeclareMathOperator{\Hom}{Hom}
\DeclareMathOperator{\Ext}{Ext}
\DeclareMathOperator{\Tor}{Tor}
\newcommand{\Hilb}{\operatorname{Hilb}}
\newcommand{\inhom}{{\mathcal H}{\it om}}
\newcommand{\inext}{{\mathcal E}{\it xt}}

\newcommand*{\ponto}{\makebox[1.5ex]{\textbf{$\cdot$}}}

\input{xypic.tex} \xyoption{all}

\usepackage{float}
\setcellgapes{3pt}

\newtheorem{coro}{Corollary}[section]
\newtheorem{defi}[coro]{Definition}

\newtheorem{lema}[coro]{Lemma}
\newtheorem{prop}[coro]{Proposition}
\newtheorem{rem}[coro]{Remark}
\newtheorem{teo}[coro]{Theorem}

\newtheorem{mcor}{Corollary}

\newtheorem{innerthmintro}{Theorem}
\newenvironment{thmintro}[1]
{\renewcommand\theinnerthmintro{#1}\innerthmintro}
{\endinnerthmintro}

\usepackage[utf8]{inputenc}

\usepackage{hyperref}
\hypersetup{
	colorlinks=true,
	linkcolor=blue,
	filecolor=magenta, 
	urlcolor=cyan,
}
\makeatletter
\@namedef{subjclassname@2020}{%
 \textup{2020} Mathematics Subject Classification}
\makeatother

\usepackage[foot]{amsaddr}
\author[M. Corrêa]{Maur\'icio Corr\^ea}
\address[M. Corrêa]{
Universit\`a degli Studi di Bari, 
Via E. Orabona 4, I-70125, Bari, Italy
\newline
\indent
UFMG\\
Avenida Ant\^onio Carlos, 6627\\
30161-970 Belo Horizonte\\ Brazil}
\email[M. Corrêa]{mauriciomatufmg@gmail.com,mauriciojr@ufmg.br, mauricio.barros@uniba.it}

\author[M. Jardim]{Marcos Jardim}
\address[M. Jardim]{IMECC - UNICAMP \\ Departamento de Matem\'atica \\
Rua S\'ergio Buarque de Holanda, 651\\ 13083-970 Campinas-SP, Brazil}
\email[M. Jardim]{jardim@unicamp.br}

\author[A. Muniz]{Alan Muniz}
\address[A. Muniz]{ Institut de Mathématiques de Bourgogne, UMR 5584 CNRS, Université de Bourgogne et Franche-Comté, 9 Avenue Alain Savary, BP 47870, 21078 Dijon
Cedex, France}
\email[A. Muniz]{alan.muniz@u-bourgogne.fr}



\date{30 January 2022}

\title{Moduli of Distributions via Singular Schemes}
\thanks{MC was partially supported by CNPq grant numbers 202374/2018-1, 302075/2015-1, 400821/2016-8. 
MJ is partially supported by the CNPQ grant number 302889/2018-3 and the FAPESP Thematic Project 2018/21391-1. 
AM is supported by a CAPES post-doctoral grant under the CAPES-Cofecub projet number 88887.191919/2018-00.
We also acknowledge the financial support from Coordenação de Aperfeiçoamento de Pessoal de Nível
Superior (CAPES) under the Finance Code 001.}

\begin{document}

\begin{abstract}
Let $X$ be a smooth projective variety. We show that the map that sends a codimension one distribution on $X$ to its singular scheme is a morphism from the moduli space of distributions into a Hilbert scheme. We describe its fibers and, when $X = \pn$, compute them via syzygies. As an application, we describe the moduli spaces of degree 1 distributions on $\p3$. We also give the minimal graded free resolution for the ideal of the singular scheme of a generic distribution on $\p3$.
\end{abstract}

\maketitle
\tableofcontents

\section{Introduction}
The theory of distributions has its origins in the studies of nineteenth-century mathematicians such as Clebsch, Darboux, Frobenius, Grassmann and Jacobi. They were motivated by the fundamental work of Pfaff, who proposed a geometric approach to the study of differential equations, see \cite{For,Hawkins}. When a distribution is integrable, in Frobenius's sense, it defines a foliation, whose study also goes back to the qualitative theory of differential equations initiated by Poincaré, Darboux and Painlevé. Since then, these theories have developed into many distinct interesting active areas of research.

In this paper we are interested in the study of families of distributions on algebraic varieties. A distribution $\sF$ on a scheme $X$, over a field $\kappa$, may be defined by a subsheaf $T_\sF\hookrightarrow TX$, called the tangent sheaf of $\sF$. In the integrable case, this has been initiated by Jouanolou \cite{Jou} with the description of the spaces of codimension one foliations on complex projective spaces, of degrees $0$ and $1$, see Section \ref{sec:fam} for the definition of degree. Later Cerveau and Lins Neto \cite{CLN} described the case of degree $2$. Recently, a partial classification in degree $3$ was obtained in \cite{CLPdeg3}. These works rely on a naive definition of moduli.

A systematic study of flat families of (not necessarily integrable) distributions was recently initiated in \cite{CCJ}, see also \cite{QUAL}. The authors of \cite{CCJ} showed that there exists a quasi-projective variety $\dd^{P}(X)$ that parametrizes isomorphism classes of distributions, on a complex projective manifold $X$, whose tangent sheaves have a fixed Hilbert polynomial $P$; a more detailed description of $\dd^{P}(X)$ is given in Section \ref{sec:fam} below, where we also address some technical issues. In addition, they successfully described certain moduli spaces of codimension one distributions, see \cite[Section 11]{CCJ}, using a forgetful map to the moduli space of sheaves. 

Each distribution $\sF$ has a (possibly empty) singular scheme $\sing(\sF)$. If $\sF$ has codimension one, the Hilbert polynomial of $\sing(\sF)$ is determined by that of $T_\sF$. This defines a set-theoretic map from $\dd^{P}(X)$, for $P$ of degree $\dim X -1$, to a Hilbert scheme. One goal of this work is to establish that this map is a morphism of schemes. More precisely, we prove the following:

\begin{thmintro}{A}\label{moduli-map}
Let $X$ be a smooth projective scheme over an algebraically closed ground field $\kappa$ and fix $H \in \Pic(X)$ ample. Let $\mathcal{D}^P(X)$ be the scheme parameterizing codimension one distributions (up to isomorphism) on $X$ with tangent sheaves having Hilbert polynomial equal to $P$. Then there exists a polynomial $Q$ depending on $P$ and a morphism 
$$
\Sigma\colon \mathcal{D}^P(X)\longrightarrow {\rm Hilb}^Q(X),
$$ 
which assigns to each distribution its singular scheme. Moreover, if ${\rm Pic}(X)\simeq \mathbb{Z}$ then $Q(t) =P_{\mathcal{O}_X}(t)-P_{TX}(t-c)+P(t-c)$, with $c \in \mathbb{Z}$ depending on $P$, and:

\begin{itemize}
    \item For any $\sF\in\mathcal{D}^P(X)$, the fiber of the morphism $\Sigma$, parameterizing distributions with singular scheme equal to $Z\in\Hilb^{Q}(X)$ is a Zariski open subset of $\mathbb{P}\left(H^0(\Omega^1_{X}(c)\otimes\mathcal{I}_Z)\right)$;
    \item When $X=\pn$ and $Z$ is not contained in a hypersurface of degree less than or equal to $c-2$, then the elements of  $H^0(\Omega^1_{\pn}(c)\otimes\mathcal{I}_Z)$ can be obtained from the linear first syzygies of the homogeneous ideal associated with $Z$.
\end{itemize}
\end{thmintro}
This result will be proved in Theorem \ref{sing:map}, Proposition \ref{fibra}, Proposition \ref{psyz} and Remark \ref{algo} below. The properties of the morphism $\Sigma$ are related to two important problems in the theory of holomorphic distributions, $\kappa = \mathbb{C}$:
\begin{enumerate}
\item Determination of the distributions by their singular schemes; 
\item Construction of distributions with prescribed singular schemes.
\end{enumerate}
The first is equivalent to the injectivity of $\Sigma$, the second to its surjectivity.

Problem (i) was considered by Gomez-Mont and Kempf \cite{GK} for foliations by curves in projective spaces, with reduced, zero dimensional singular schemes. Campillo and Olivares showed in \cite{CO2} that the hypothesis that the singular scheme is reduced may be removed. Araujo and Corr\^ea \cite{CoA} provided sufficient conditions for distributions of arbitrary dimension to be uniquely determined by their singular schemes. 

On the other hand, Problem (ii) was first considered by Campillo and Olivares in \cite{CO} in order to understand under which conditions a subscheme of $\p2$ can be the singular subscheme of a foliation; they provided a Cayley--Bacharach type theorem for codimension one foliations on $\p2$. The same problem for the case of codimension one foliations in $\p3$ also has been proposed by Cerveau in \cite{Ce}.

As a first application, we consider generic codimension one distributions on $\pn$, that is, those distributions with 0-dimensional singular schemes. It was proved in \cite[Theorem 1.4]{CoA} that if $d \geq n$, then a generic distribution of degree $d$ on $\pn$ with zero dimensional singularities is determined by its singular scheme; however, a careful reader will notice that their argument works in greater generality; in fact, it holds for every $d$ and $n$, except for $d=1$ and $n=2,3$. Degree 1 distributions (foliations) on $\mathbb{P}^2$ are not determined by their singular schemes, as one can easily define non isomorphic foliations with three nondegenerate singularities in general position, see \cite[p. 882]{CO:plane}. 
For $n=3$ we will use our syzygetic approach to describe distributions with a given 0-dimensional singular scheme.

\begin{thmintro}{B}\label{mthm2}
Let $\sF $ be a degree $d\geq 1$ distribution on $\mathbb{P}^3$. If $Z$, the singular scheme of $\sF$, is zero-dimensional then the homogeneous ideal $I(Z)$ has minimal graded free resolution given by
$$
0 \longrightarrow S(-3d-2) \longrightarrow S(-2d-1)^4\oplus S(-d-2) \longrightarrow S(-d-1)^4 \oplus S(-2d) \longrightarrow I(Z) \longrightarrow 0 
$$
where $S = \kappa[x_0, \dots, x_3]$. 
Moreover, the space of distributions singular at $Z$ has dimension $4$ if $d=1$ and, if $d\geq 2$ then $\sF $ is the unique degree $d$ distribution singular at $Z$. 

\end{thmintro}

In particular, we obtain the following corollary.

\begin{mcor}\label{determ-generic}
Let $\sF$ be a degree $d$ distribution $\mathbb{P}^n$ with isolated singularities. Then  $\sF $ is determined by its singular scheme if $(d,n) \notin \{(1,2),(1,3)\}$.
\end{mcor}

Theorem \ref{moduli-map} can also be used to give a full description of the irreducible components of the moduli space of isomorphism classes of codimension one distributions of degree 1 on $\p3$. A classification of the possible Chern classes of the tangent sheaf $T_\sF$ of such a distribution was given in \cite{CCJ}. In fact, it was shown that $\sing(\sF)$ is never planar, and $T_\sF$ either splits as a sum of line bundles or is $\mu$-stable; the relevant information is summarized in Table \ref{table deg 1 v2}.

\begin{table}[h]
\begin{tabular}{|c | c | c | c | c | } \hline
 $c_2(T_\sF)$ & $c_3(T_\sF)$ & $T_\sF$ & $\sing(\sF)$ \\ \hline\hline
3 & 5 & stable & 5 points \\ \hline
2 & 2 & stable & line and 2 points \\ \hline
1 & 1 & stable & conic and 1 point \\ \hline
0 & 0 & split & twisted cubic \\ \hline
\end{tabular}
\medskip
\caption{Classification of codimension one distributions of degree 1 on $\p3$, according to \cite[Section 8]{CCJ}; the last column describes the singular scheme of a generic distribution whose tangent sheaf has the corresponding Chern classes.}
\label{table deg 1 v2}
\end{table}

Let $\dd(a,b)$ denote the moduli space of degree 1 distributions $\ff$ on $\p3$ such that $(c_2(T_\ff),c_3(T_\ff))=(a,b)$. We prove:

\begin{thmintro}{C}
\label{mthm3} 
The moduli space $\dd(a,b)$ is an irreducible smooth quasi-projective variety whenever it is non-empty. In addition:
\begin{itemize}
\item the image of the morphism $\Sigma\colon\dd(0,0)\rightarrow {\rm Hilb}^{3t+1}(\p3)$ is an open subset of the aCM component of ${\rm Hilb}^{3t+1}(\p3)$, and its fibres are open subsets of $\p1$; $\dim\dd(0,0)=13$.
\item the morphism $\Sigma\colon\dd(2,2)\rightarrow {\rm Hilb}^{t+3}(\p3)$ is dominant, and its fibres are open subsets of $\p5$; $\dim\dd(2,2)=15$.
\item the image of the morphism $\Sigma\colon\dd(1,1)\rightarrow {\rm Hilb}^{2t+2}(\p3)$ is an open subset of the non pure component of ${\rm Hilb}^{2t+2}(\p3)$, and its fibres are open subsets of $\p3$; $\dim\dd(1,1)=14$.
\end{itemize}
\end{thmintro}

Further detailed information about the moduli spaces, the image of the morphism $\Sigma$ as well as its fibres can be found in the Theorems \ref{cubic}, \ref{line+2pt}, \ref{conic+pt} below. The proof of smoothness uses technical lemmata established in the Appendix \ref{Apdx} and which might be of independent interest.

We point out that the fact that $\dd(1,1)$ is an irreducible, smooth quasi-projective variety of dimension $14$ was already established in \cite[Proposition 11.2]{CCJ} by analyzing the corresponding forgetful morphism. Let us also mention, for the sake of completeness, that $\dd(3,5)$, being the moduli space of generic distributions of degree 1, is an open subset of $\p{19}$.

Finally, we need to mention that the computations in this article were done, in part, with the help of the computer algebra system Macaulay2 \cite{M2}. Besides speeding up the work, it was crucial to building up intuition to achieve some of our results.

\addtocontents{toc}{\protect\setcounter{tocdepth}{0}}
\subsection*{Acknowledgments}
We thank Renato Vidal Martins for his participation in the initial stages of this project and many enlightening discussions. We wish to thank the referee for many important suggestions.

\addtocontents{toc}{\protect\setcounter{tocdepth}{2}}

\section{Families of Distributions} \label{sec:fam}

Let $\kappa$ be an algebraically closed field and let $X$ be a smooth projective variety over $\kappa$. A \emph{codimension $r$ distribution} $\sF$ on $X$ is given by an exact sequence
\begin{equation}\label{eq:Dist}
\mathscr{F}:\ 0 \longrightarrow T_\sF \stackrel{\phi}{ \longrightarrow} TX \stackrel{\pi}{ \longrightarrow} N_{\sF} \longrightarrow 0,
\end{equation}
where $T_\sF$ is a coherent sheaf of rank $m:=\dim(X)-r$, and $N_{\sF}$ is a torsion free sheaf. The sheaves $T_\sF$ and $N_{\sF}$ are called the \emph{tangent} and the \emph{normal} sheaves of $\mathscr{F}$, respectively. Note that $T_\sF$ must be reflexive \cite[Proposition 1.1]{HL}.

Two distributions $\mathscr{F}$ and ${\mathscr{F}}'$ are isomorphic if there exists an isomorphism $\beta\colon T_{\mathscr{F}}\rightarrow T_{\mathscr{F}'}$ such that $\phi'\circ\beta=\phi$. By taking exterior power $\bigwedge^{m}\phi \colon \det(T_{\mathscr{F}})\rightarrow\bigwedge^{m}TX$, every codimension $r$ distribution induces a section of $H^0(X,\bigwedge^{m}TX \otimes \det(T_\sF)^{\vee})$. If $\mathscr{F}$ and ${\mathscr{F}}'$ are isomorphic distributions, then $\bigwedge^{m}\phi'=\lambda \cdot \bigwedge^{m}\phi$, for some constant $\lambda\in\kappa^*$. Therefore, every isomorphism class of codimension $r$ distribution on $X$ induces an element in the projective space
$$
\mathbb{P}\!\left(H^0\!\left(\bigwedge^{m}TX \otimes \det(T_\sF)^{\vee}\right)\right) = 
\mathbb{P}\!\left(H^0\!\left(\Omega^r_X \otimes\det(TX)\otimes \det(T_\sF)^{\vee}\right)\right). 
$$

The \emph{singular scheme} of $\mathscr{F}$ is defined as follows. Consider the the dual morphism $\phi^\vee\colon \Omega^1_X\rightarrow T_\sF^\vee$. Taking the maximal exterior power and a twist we obtain 
$$\Omega^{m}_X\otimes \det(T_\sF) \longrightarrow \mathcal{O}_X;$$ 
the image of such a morphism is the ideal sheaf $\iz$ of a subscheme $\sing(\sF):=Z\subset X$, called the singular scheme of $\mathscr{F}$.

Next, we examine the notion of a family of distributions parameterized by a scheme $S$ of finite type over $\kappa$. In fact, the authors of \cite{QUAL} and \cite{CCJ} have proposed two slightly different notions for a \emph{family of distributions parameterized by $S$}, namely:

\begin{defi}\label{def:quall}
According to \cite{QUAL}, a family of distributions parameterized by $S$ is given by an $S$-scheme $\mathcal{X}\rightarrow S$ and a short exact sequence
\[
0 \longrightarrow \mathbf{F} \longrightarrow T\mathcal{X}_S \longrightarrow \mathbf{N} \longrightarrow 0,
\]
where $T\mathcal{X}_S$ stands for the relative tangent sheaf. Moreover, the family is called flat over $S$ if so is $\mathbf{N}$.
\end{defi}

\begin{defi}\label{def:CCJ} 
Let $X$ be a smooth projective variety and let $TX_S$ denote the relative tangent sheaf of $X\times_\kappa S \rightarrow S$. According to \cite{CCJ}, a family of distributions on $X$ parameterized by $S$ is given by a subsheaf $\boldsymbol{\phi} \colon \mathbf{F} \hookrightarrow TX_S$ such that for each $s\in S$ the restriction $\boldsymbol{\phi}_s\colon \mathbf{F}|_s \rightarrow TX_s$ is injective and $\operatorname{coker}(\phi_s)$ is torsion free.
\end{defi}

Throughout this work we will stick with the latter definition but, in fact, we can reconcile both proposals. In order to achieve that we must make some technical considerations. Recall the following consequence \cite[Exposé IV, p.81]{SGA1} of the Local Flatness Criterion, see also \cite[p.35]{HL}.

\begin{lema}
\label{cor:injfib}
Let $(A,\mathfrak{m}) \rightarrow (B,\mathfrak{n})$ be a morphism of Noetherian local rings $(\mathfrak{m}B \subset \mathfrak{n})$ and let $u\colon M' \rightarrow M$ be a morphism of finitely generated $B$-modules. Suppose that $M$ is flat over $A$. Then the following are equivalent:
\begin{enumerate}
	\item $u$ is injective and $\operatorname{coker}(u)$ is flat over $A$;
	\item $u\otimes 1 \colon M'\otimes_A A/\mathfrak{m} \rightarrow M\otimes_A A/\mathfrak{m}$ is injective.
\end{enumerate}
\end{lema}

Suppose that $S$ and $X$ are noetherian and $X$ is flat over $S$. As any coherent sheaf $\mathcal{N}$ on $X$ is the quotient of a locally free sheaf, 
\[
0 \longrightarrow \ker(\phi) \longrightarrow E \stackrel{\phi}{\longrightarrow} \mathcal{N} \longrightarrow 0.
\]
Then Lemma \ref{cor:injfib} tells us that $\mathcal{N}$ is flat over $S$ if and only if the sequence above base changes correctly, i.e., for every $s\in S$ we have 
\[
\ker(\phi)|_s \simeq \ker(\phi_s).
\]
Note that we have only used that $X$ is flat over $S$ to imply that $E$ being locally free (hence flat over $X$) is also flat over $S$.

\begin{prop}
 A family of distributions according to Definition \ref{def:CCJ} is equivalent to a flat family in Definition \ref{def:quall} such that $\mathcal{X} = X\times_\kappa S$ and $\mathbf{N}_s$ is torsion free for every $s\in S$.
\end{prop}

\begin{proof}
Since $X$ is trivially flat over $\operatorname{Spec}(\kappa)$ we know that $X\times S$ is flat over $S$ since flatness has base change property. Also note that both are noetherian since they are of finite type over $\kappa$. Now $TX_S$ is a locally free sheaf on $X\times S$ hence flat over $X\times S$ and, by transitivity of flatness \cite[Proposition 9.1A]{HART:AG}, $TX_S$ is flat over $S$. Then we apply Lemma \ref{cor:injfib} to the morphism $\boldsymbol{\phi}\colon\mathbf{F} \rightarrow TX_S$ to show that $\coker(\boldsymbol{\phi})$ is flat over $S$ if and only if $\boldsymbol{\phi}_s\colon \mathbf{F}|_s \rightarrow TX_s$ is injective for every $s\in S$.
\end{proof}

Note that since $TX_S$ is flat over $S$, $\coker(\boldsymbol{\phi})$ being flat implies that $\mathbf{F}$ is flat over $S$ as well. 


\section{Moduli of Codimension One Distributions and Hilbert Schemes}\label{sec:moduli}
We now recall the definition of the moduli space of distributions, which is based on the notion of family of distributions explained in the previous section.
Let $X$ be a smooth projective variety and fix $H=\mathcal{O}_X(1)$ an ample line bundle (a polarization). We will write $\mathcal{O}_X(k)$, $k\in \mathbb{Z}$, for its tensor powers.

Let $\sch$ denote the category of schemes of finite type over $\kappa$, and $\sets$ be the category of sets. Fix a polynomial $P\in\mathbb{Q}[t]$, and consider the functor
\begin{gather*}
    \dist^P_X \colon \sch^{\rm op} \longrightarrow \sets , \\
   \dist^P_X(S) := \left\{(\mathbf{F},\boldsymbol{\phi}) \mid \text{\parbox[c]{.4\linewidth}{\raggedright \ $\mathbf{F} \stackrel{\boldsymbol{\phi}}\hookrightarrow T\mathcal{X}_S$  is a family of distributions with  $\chi(F|_s(t)) = P(t), \, \forall s\in S$ }} \ \right\}/\sim, 
\end{gather*}
where a family is considered as in Definition \ref{def:CCJ}, and we say that $(\mathbf{F},\boldsymbol{\phi})\sim(\mathbf{F}',\boldsymbol{\phi}')$ if there exists and isomorphism $\boldsymbol{\beta}\colon\mathbf{F}\rightarrow\mathbf{F}'$ such that $\boldsymbol{\phi}=\boldsymbol{\phi}'\circ\boldsymbol{\beta}$.

One can show that the functor $\dist^P_X$ is represented by a quasi-projective scheme $\mathcal{D}^P(X)$, that is, there exists an isomorphism of functors $\dist^P_X \stackrel{\sim}{\rightarrow} \Hom(\ponto,\mathcal{D}^P(X))$ \cite[Proposition 2.4]{CCJ}. Furthermore, the scheme $\mathcal{D}^P(X)$ is an open subset of the quot scheme ${\rm Quot}^{P_{TX}-P}(TX)$. This last observation implies that the Zariski tangent space of $\mathcal{D}^P(X)$ at the point $[F,\phi]$ is $\Hom(F,\coker\phi)$ and leads to the following smoothness criterion as an immediate consequence of \cite[Proposition 2.2.8]{HL}

\begin{prop} \label{smooth}
Let $[\ff]\in\mathcal{D}^P(X)$ be the isomorphism class of a distribution on $X$. If \linebreak $\Ext^1(T_\ff, N_\ff)=0$, then $\mathcal{D}^P(X)$ is smooth at $[\ff]$ and $\dim T_{[\ff]}\mathcal{D}^P(X)=\dim \Hom(T_\ff, N_\ff)$.
\end{prop}

Unfortunately, the vanishing condition $\Ext^1(T_\ff, N_\ff)=0$ is not satisfied by most of the codimension one distributions of degree 1, treated in Section \ref{sec:deg1} below. For these examples we compute $\dim\mathcal{D}^P(X)$ and $\dim \Hom(T_\ff, N_\ff)$ separately and show that they match.

We are finally in position to prove our first main result.

\begin{teo}\label{sing:map}
Let $X$ be a smooth projective scheme over an algebraically closed ground field $\kappa$ and fix $H \in \Pic(X)$ ample. Let $\mathcal{D}^P(X)$ be the scheme parameterizing codimension one distributions (up to isomorphism) on $X$ with tangent sheaves having Hilbert polynomial equal to $P$. Then there exists a polynomial $Q$ depending on $P$ and a morphism 
$$
\Sigma\colon \mathcal{D}^P(X)\longrightarrow {\rm Hilb}^Q(X),
$$ 
which assigns to each distribution its singular scheme. 
\end{teo}

\begin{proof}
We claim that there exists a natural transformation between the functors $\mathcal{D}ist^P_X$ and $\mathcal{H}ilb^Q_X$. In fact, for each $S\in \mathfrak{Sch}_\kappa$, we have that an element of $\mathcal{D}ist^P_X(S)$ defines an exact sequence
$$
0 \longrightarrow F\stackrel{\phi}{\longrightarrow}TX_S \longrightarrow N \longrightarrow 0
$$
of coherent sheaves on $X\times S$, such that for each $s\in S$, the sequence
\begin{equation}\label{seq:distrest}
    0 \longrightarrow F_s\stackrel{\phi_s}{\longrightarrow}TX \longrightarrow N_s \longrightarrow 0 
\end{equation}
is a distribution on $X$ and $F_s$ has Hilbert polynomial (with respect to $H$) equal to $P$. As $F_s$ is of codimension $1$ we have that
$N_s$ is a rank one torsion free sheaf. As $X$ is smooth, $N^{\vee\vee}$ is locally free owing to \cite[Lemma 6.13]{Kol}. Moreover, $N^{\vee\vee}$ must coincide with $\det(N)$ whose formation commutes with base change \cite[p.39]{HL}.

Since $N_s$ is torsion free for every $s\in S$, due to Lemma \ref{cor:injfib}, we have an injection $N\hookrightarrow N^{\vee \vee}$ hence we have
\[
0 \longrightarrow N \otimes N^{\vee} \longrightarrow N^{\vee \vee} \otimes N^{\vee} = \mathcal{O}_{X\times S}   \longrightarrow \mathcal{O}_{Z} \longrightarrow 0 
\]
where $Z \subset X\times S$ is flat over $S$, and $Z_s$ is the singular scheme of the distribution defined by \eqref{seq:distrest}. This construction commutes with base change and thus defines a morphism $\Sigma \colon \mathcal{D}^P(X)\longrightarrow {\rm Hilb}^Q(X)$, where $Q$ is a polynomial depending on $P$.

\end{proof}

In the study of the morphism $\Sigma$ it is useful to understand its fibers. Let $X$ and $\sF \in \mathcal{D}^P(X)$ be as in Theorem \ref{sing:map}, and set $Z = \sing(\sF )$. In addition, assume that $\Pic(X) = \mathbb{Z}\cdot H$ where $H$ is the ample generator; we then have that $Q(t)=P_{\mathcal{O}_X}(t)-P_{TX}(t+c)+P(t+c)$, where $c H = c_1(TX) - c_1(T_{\sF})$.

It then follows that any distribution on $X$ singular at $Z$ is defined by a twisted $1$-form $\omega \in H^0(\Omega^1_X(c) \otimes \mathcal{I}_Z)$. In particular, the fiber of $\Sigma$ over $Z$ lies inside $$\mathbb{P}(H^0(\Omega^1_X(c) \otimes \mathcal{I}_Z)).$$ A priori, the singular scheme of $\omega$ could be larger, even of codimension one. However, the locus of $1$-forms vanishing only on $Z$ is a (nonempty) Zariski open set, thanks to the flattening stratification \cite[Chapter 8]{MUM:LCAS}. 

\begin{prop} \label{fibra}
Let $X$ be a smooth projective scheme over an algebraically closed field $\kappa$. Assume further that ${\rm Pic}(X)\simeq \mathbb{Z}$. Let $\sF$ be a codimension one distribution on $X$. Then the fiber
$$ \Sigma^{-1}\left(\Sigma(\sF ) \right)\subset \mathbb{P}\!\left(H^0(\Omega^1_{X}(c)\otimes\mathcal{I}_Z)\right) $$
parameterizing distributions with singular scheme equal to $Z$ is Zariski open.
\end{prop}

\begin{proof}
Let $S = \mathbb{P}\left(H^0(X, \Omega^1_{X}(c) \otimes \mathcal{I}_Z)\right)$ and consider $\omega_S \in H^0\!\left(X_S, \Omega^1_{X}(c) \boxtimes \mathcal{O}_S(1)\right)$ the ``tautological" $1$-form - here $\boxtimes$ means the external tensor product. Then we have the family $\mathcal{Z}$ of subschemes defined by 
\[
 TX(-c) \boxtimes \mathcal{O}_S(-1)\stackrel{\omega_S}{\longrightarrow} \mathcal{O}_{X_S} \longrightarrow \mathcal{O}_\mathcal{Z} \longrightarrow 0
\]
Note that for each point $s\in S$ the fiber $\mathcal{Z}_s$ is the singular scheme of $\omega_s= \omega_S|_{X_s}$, the restriction of $\omega_S$ to the fiber over $s$. We may apply the flattening stratification to $\mathcal{O}_\mathcal{Z}$, \cite[Corollary, p.60]{MUM:LCAS}. 

It follows that the map $s \mapsto \chi(\mathcal{O}_{\mathcal{Z}_s}(t))$ is locally constant. We denote its finite set of values, by $N$. Then each $f \in N$ corresponds to a stratum $S_f$ whose underlying set $|S_f|$ is composed by the closed points $s\in S$ such that $ \chi(\mathcal{O}_{\mathcal{Z}_s}(t)) = f(t)$. Moreover the Zariski closure of $|S_f|$ is contained in $\cup_{g\geq f} |S_g|$, where $f>g$ if and only if $f(t) > g(t)$ for $t\gg 0$. 

Let $p$ be the Hilbert polynomial of $Z$. Note $p\in N$ since $Z$ is the singular scheme of $\sF $. 
We have that $Z\subset \mathcal{Z}_s$ for every $s\in S$ (closed) hence $p$ is minimal in $N$. Therefore $S_p\subset S$ is a nonempty open subscheme. 
\end{proof}

Note that the tautological $1$-form restricted to $S_p$ induces the universal family of the fiber:
\[
0 \longrightarrow T_\sF \longrightarrow TX_{S_p} \stackrel{\omega_{S_p}}{\longrightarrow} \mathcal{I}_\mathcal{Z}\otimes \mathcal{O}_{S_p}(1)\boxtimes \mathcal{O}_X(c) \longrightarrow 0 
\]

\begin{rem}
Still assuming that  $\Pic(X) = \mathbb{Z}$, consider $U := \Sigma \left( \mathcal{D}^P(X)\right)$, the image of $\Sigma$. Proposition \ref{fibra} shows that the fiber over each point of $U$ is an open subset of a projective space. In addition, let $\mathcal{U}$ denote the restriction of the universal ideal sheaf on $X\times\Hilb^Q(X)$ to $X\times U$; one can then consider the projective scheme ($\pi_X$ and $\pi_U$ are the natural projections)
$$ \mathcal{P} := \mathbb{P}\!\left( {\pi_U}_*\left(\mathcal{U} \otimes \pi_X^\ast\Omega^1_X(c)\right) \right) $$
which contains $\mathcal{D}^P(X)$ as a subset. 
The natural question that arises is whether $\mathcal{D}^P(X)$ is included in $\mathcal{P}$ as an open subscheme. Note that if $h^i(\Omega^1_X(c)\otimes \iz)= 0$ for $i>0$ and every $Z\in U$, then $\mathcal{P}$ is the projectivization of a vector bundle over $U$, see \cite[III, Theorem 12.11]{HART:AG}.
\end{rem}

\section{Syzygies and Codimension one Distributions on \texorpdfstring{$\mathbb{P}^n$}{Pn}}
\label{secprn}
The degree of a codimension one distribution $\sF$ on $\pn$ is geometrically defined by the tangency between $\sF$ and a general line $L$. If a global section of $\Omega_\pn^1(k)$ defines $\sF$ then its restriction to $L$ defines a degree $k-2$ divisor (on $L$), called the tangency divisor. Thus the degree is defined by $d=k-2$, see \cite[section 2.5]{CJV}. 
In this particular case, the sequence \eqref{eq:Dist} becomes
\begin{equation}\label{seq:distpn}
    0 \longrightarrow T_\sF \stackrel{\phi}{ \longrightarrow} \tpn  \longrightarrow \iz(d+2) \longrightarrow 0.
\end{equation}

Given the singular scheme $Z$ of a degree $d$ distribution on $\mathbb{P}^n$, the distributions singular at $Z$ are defined by twisted $1$-forms vanishing on $Z$. We will now show a correspondence between these forms and linear first syzygies of the homogeneous ideal $I(Z)$, provided that $Z$ is not contained in a hypersurface of degree $\leq d$.

As a consequence of the classification obtained in \cite[Section 8]{CCJ}, every codimension one distribution of degree 1 on $\p3$ satisfies $h^1(T_\sF(-2))=0$. Thus \eqref{seq:distpn} implies that $h^0(\iz(d))=h^1(T_\sF(-2))=0$, i.e., $Z= \sing(\sF)$ cannot be contained in a hypersurface of degree $\leq d$. We prove below that this condition is also verified for codimension one distributions on $\pn$ with zero-dimensional singular scheme, which we call generic.

\begin{defi}
A codimension one distribution on $\pn$ is said to be \emph{generic} if its singular scheme is either empty or a zero-dimensional subscheme of $\pn$. 
\end{defi}

The previous definition is justified by the fact that codimension one distributions of degree $d$ with zero-dimensional singular scheme form an open subset of $\mathbb{P}(H^0(\opn(d+2)))$.

\begin{lema}\label{lem:generic}
Let $Z$ be a zero-dimensional subscheme of $\mathbb{P}^n$. If $Z$ is the singular scheme of a degree $d$ codimension one distribution on $\pn$ then $Z$ is not contained in a hypersurface of degree $\leq d$.
\end{lema}

\begin{proof}
Since there are no hypersurfaces of degree zero, the conclusion holds vacuously for $d\leq 0$; we may assume $d\geq1$. Let $\mathcal{I}_Z$ be the ideal sheaf of $Z$. To prove that $Z$ is not contained in a hypersurface of degree at most $d$ we will show that $H^0(\mathcal{I}_Z(d))= \{0\}$.

Since $\mathcal{I}_Z$ is the image of a morphism $\omega \colon \tpn(-d-2) \longrightarrow \mathcal{O}_{\mathbb{P}^n}$, the hypothesis on the dimension implies that $Z$ is locally a complete intersection, thus the associated Koszul complex is exact. Then we may use it to compute $H^0(\mathcal{I}_Z(d))$. Twisting the Koszul resolution by $\mathcal{O}_{\mathbb{P}^n}(d)$ and splitting it into short exact sequences yields the following.
\begin{align*}
 0 \longrightarrow \mathcal{O}_{\mathbb{P}^n}((d+1)(1-n)) \longrightarrow &\bigwedge\limits^{n-1}\tpn(-(n-1)(d+2)+d) \longrightarrow F_{n-2} \longrightarrow 0 \\
 & \vdots \\
 0 \longrightarrow F_j \longrightarrow &\bigwedge\limits^j \tpn(-j(d+2)+d) \longrightarrow F_{j-1} \longrightarrow 0 \\
 &\vdots \\
 0 \longrightarrow F_1 \longrightarrow &\tpn(-2) \longrightarrow \mathcal{I}_Z(d) \longrightarrow 0
\end{align*}
Due to Bott's formulas \cite[p.4]{OSS} and Serre duality, we have that for $1 \leq j \leq n-1$ and $d\geq 1$, 
$$
H^{j-1}\!\left(\bigwedge^j \tpn(-j(d+2)+d)\right) = \{0\}.
$$
Then the long exact sequences of cohomology yields the following injections
\[
H^0(\mathcal{I}_Z(d))\hookrightarrow H^1(F_1) \hookrightarrow \dots \hookrightarrow H^{n-2}(F_{n-2}) \hookrightarrow H^{n-1}(\mathcal{O}_{\mathbb{P}^n}((d+1)(1-n)) ) = \{0\}.
\]
This completes the proof.
\end{proof}

\begin{rem}
The previous lemma has an alternative proof in the case $n=3$. Let 
$$ \sF ~:~ 0 \longrightarrow T_\sF \longrightarrow \tp3 \longrightarrow \iz(d+2) \longrightarrow 0 $$
be a generic codimension one distribution of degree $d$ on $\p3$; set $Z:=\sing(\sF)$. Since $\dim Z=0$, we have that $\inext^1(\iz,\op3)=0$. Thus dualizing the exact sequence above, we obtain the exact sequence
$$ 0 \longrightarrow \op3(-d-2) \longrightarrow \Omega^1_{\mathbb{P}^3} \longrightarrow T_\sF(d-2) \longrightarrow 0 $$
where we used the identity $T_\sF^\vee\simeq T_\sF(d-2)$ valid for rank 2 reflexive sheaves. Using these two sequences we have that $h^0(\iz(d))=h^1(T_\sF(-2))=0$, as desired. 
\end{rem}

Furthermore, we can also check that if $\sF$ is a codimension one distribution on $\pn$ whose tangent sheaf splits as a sum of line bundles, then $\sing(\sF)$ is not contained in a hypersurface of degree $d$. Indeed, if $T_\sF$ splits then  $h^1(T_\sF(-2))=0$. 

\begin{rem}
This discussion leads to the following natural question: does there exist a degree $d$ codimension one distribution $\sF$ on $\pn$ such that $\sing(\sF)$ is contained in a hypersurface of degree $d$? 
\end{rem}

\begin{prop}\label{psyz}
Let $Z \subset \mathbb{P}^n$ be a closed subscheme and let $d\geq 1$ be an integer. Suppose that $Z$ is not contained in a hypersurface of degree less than or equal to $d$. Then there exists a linear isomorphism between the spaces of degree $d+2$ twisted $1$-forms singular at $Z$ and linear first syzygies of the homogeneous ideal $I(Z)$. 
\end{prop}

\begin{proof}
The space of degree $d+2$ twisted $1$-forms singular at $Z$ is $H^0( \Omega^1_{\mathbb{P}^n}(d+2) \otimes \mathcal{I}_Z) $, where $\mathcal{I}_Z$ is the ideal sheaf of $Z$. Consider the saturated homogeneous ideal is $I(Z) = \bigoplus_j H^0(\mathcal{I}_Z(j))$. By hypothesis, $H^0(\mathcal{I}_Z(j)) = \{0\}$ for $j\leq d$ hence the space of linear first syzygies of $I(Z)$ is ${\rm Tor}^S_1(I(Z), \kappa)_{d+2}$, where $S = \kappa[x_0, \dots, x_n]$. 

Consider the second exterior power of the Euler sequence twisted by $\mathcal{I}_Z(d+2)$:
\[
0 \longrightarrow \Omega^2_{\mathbb{P}^n}(d+2) \otimes \mathcal{I}_Z \longrightarrow \mathcal{I}_Z(d)^{\oplus \binom{n+1}{2}} \longrightarrow \Omega^1_{\mathbb{P}^n}(d+2) \otimes \mathcal{I}_Z \longrightarrow 0
\]
which is exact since ${\rm Tor}_1^{\mathcal{O}_{\mathbb{P}^n}}(\Omega^1_{\mathbb{P}^n}(d+2) , \mathcal{I}_Z) = \{0\}$. From the long sequence of cohomology we have that
\[
H^0(\mathcal{I}_Z(d))^{\oplus \binom{n+1}{2}} \rightarrow H^0( \Omega^1_{\mathbb{P}^n}(d+2) \otimes \mathcal{I}_Z) \stackrel{\phi}{\longrightarrow} H^1( \Omega^2_{\mathbb{P}^n}(d+2) \otimes \mathcal{I}_Z) \rightarrow H^1(\mathcal{I}_Z(d))^{\oplus \binom{n+1}{2}}
\]
is exact. We have that $\phi$ is injective, by hypothesis, and from \cite[Theorem 5.8]{EI:SYZ5} the image of $\phi$ is precisely ${\rm Tor}^S_1(I(Z), \kappa)_{d+2}$. Thus we have the desired isomorphism.
\end{proof}

Now we present this correspondence in an effective way. Fix $Z\subset \mathbb{P}^n$ a closed subscheme with ideal sheaf $\mathcal{I}_Z$. Suppose, as above, that $h^0(\mathcal{I}_Z(d))=0$. Also fix $\{G_1, \dots, G_r\}$ a (linear) basis of $H^0(\mathcal{I}_Z(d+1))$. 

First let 
\[
\omega = F_0dx_0 + \dots + F_n dx_n 
\]
be a degree $d+2$ twisted $1$-form vanishing on $Z$, i.e., $F_0 ,\dots, F_n \in H^0(\mathcal{I}_Z(d+1))$. Then $F_j = \sum f_{ij} G_j$ for $f_{ij}\in \kappa$. It follows that
\[
0 = \iota_R \omega = \sum_{i=0}^n x_i F_i = \sum_{j=1}^r\left(\sum_{i=0}^n x_i f_{ij}\right) G_j
\]
hence $\left(\sum_{i=0}^n x_i f_{i1}, \dots , \sum_{i=0}^n x_i f_{ir}\right)$ provides a first syzygy for any set of generators of $I(Z)$ containing $\{G_1, \dots, G_r\}$. 

Conversely, fix a minimal generating set $\{G_1, \dots, G_r, H_1 \dots,H_k\}$ for $I(Z)$ such that $\deg H_j \geq d+2$ for every $j$. Then a linear first syzygy for this set of generators can be written as $(L_1, \dots L_r, 0 ,\dots ,0)$ for $L_j$ linear homogeneous polynomials. Define 
\[
\omega = \sum_{j=1}^r G_j dL_j.
\]
Note that $\iota_R \omega = \sum_{j=1}^r G_jL_j =0$ hence $\omega$ defines a twisted $1$-form on $\mathbb{P}^n$ that is clearly singular at $Z$. 

\begin{rem}\label{algo}
The discussion above leads to an algorithm to compute the generic $1$-form $\omega$ vanishing on a subscheme $Z$. It goes as follows: The input is the ideal $I(Z)$ and the output is the generic $1$-form $\omega$ singular at $Z$.
\begin{itemize}
 \item Take $G = (F_0, \dots , F_r)$ an ordered minimal generating set for $I(Z)$;
 \item Compute $M$ the matrix of linear syzygies of $G$ and let $k+1$ be its number of columns;
 \item Define additional variables $(t_0, \dots, t_k)$;
 \item Return $$\omega := G\cdot dM\cdot (t_0, \dots, t_k)^T,$$ where $dM$ stands for the entrywise exterior derivative and the dots are matrix products.
\end{itemize}
\end{rem}

To conclude this section we take a closer look at generic codimension one distributions. For $d \geq n$, such distributions are determined by their singular schemes, see \cite[Theorem 1.4]{CoA}. However a careful reader will notice that their argument works in more generality. Indeed it holds for every $d$ and $n$, except for $d=1$ and $n=2,3$.

On the other hand, we argue that generic codimension one distributions of degree 1 on $\p2$ or $\p3$ are not determined by their singular schemes. Indeed, for $n=2$ one can easily define non isomorphic foliations with three nondegenerate singularities in general position, see \cite[p. 882]{CO:plane}. For $n=3$ we will use our syzygetic approach to show that there exists a $4$-dimensional family of distributions with a given 0-dimensional singular scheme. 

\begin{teo}\label{resolution}
Let $\sF $ be a degree $d\geq 1$ distribution on $\mathbb{P}^3$. If $Z = {\rm Sing}(\sF )$ is zero-dimensional then the homogeneous ideal $I(Z)$ has minimal graded free resolution given by
$$
0 \longrightarrow S(-3d-2) \longrightarrow S(-2d-1)^4\oplus S(-d-2) \longrightarrow S(-d-1)^4 \oplus S(-2d) \longrightarrow I(Z) \longrightarrow 0 
$$
where $S = \kappa[x_0, \dots, x_3]$. 
Moreover, the space of distributions singular at $Z$ has dimension $4$ if $d=1$ and, if $d\geq 2$ then $\sF $ is the unique degree $d$ distribution singular at $Z$. 
\end{teo}

\begin{proof}
Let $\omega = \sum_{j=0}^3A_j dx_j$ be a $1$-form that defines $\sF $. We claim that $I(Z) = (A_0,A_1,A_2,A_3,P)$ where $P$ is the Pfaffian of the matrix defined by $d\omega$, i.e., $P = {\rm Pf}( J - J^T)$ for $J$ the jacobian matrix of $(A_0,A_1,A_2,A_3)$. It follows that $\deg P = 2d$ since $\deg A_j = d+1$. 

From \cite[Corollary 4.9]{MNP:B-Rim} we have that $Z$ is arithmetically Gorenstein. Then it follows from \cite[Section 4]{BE} that $I(Z)$ has a minimal resolution given by
$$
0 \longrightarrow \bigwedge^5F \longrightarrow \bigwedge^4 F \longrightarrow F \longrightarrow I(Z) \longrightarrow 0 
$$
where $F$ is a free $S$-module. Therefore $F = S(-d-1)^4 \oplus S(-2d)$ and the resolution becomes 
$$
0 \longrightarrow S(-6d-4+l) \stackrel{g^T}{\longrightarrow} S(-5d-3+l)^4\oplus S(-4d-4+l) \stackrel{f}{\longrightarrow} S(-d-1)^4 \oplus S(-2d) \stackrel{g}{\longrightarrow} I(Z) \longrightarrow 0 
$$
where $l$ accounts for the shift made so that every map has degree zero. Since $(x_0, \dots, x_3,0)$ is a first syzygy we can compute $l = 3d+2$ whence follows the statement. 

From the degrees in the resolution we can compute the degrees of the entries of the map $f$ to see that for $d\geq 2$ the only linear syzygies are multiples of $(x_0, \dots, x_3,0)$. In the case $d=1$ we have $-2d-1 = -d-2 =-3$ and $-d-1 = -2d =-2$, hence every entry of $f$ is linear. It follows from Propositions \ref{psyz} and \ref{fibra} that the space of distributions singular at $Z$ has (projective) dimension $4$ if $d=1$ and dimension $0$ (i.e. $\sF $ is unique) if $d\geq 2$.

Now we prove the claim. Let $I = (A_0,A_1,A_2,A_3)$ be the ideal generated by the coefficients of $\omega$. Recall that $I(Z) = \bigoplus_{l\in \mathbb{Z}} H^0( \iz(l))$ is the saturation of $I$. Since $\dim Z=0$, the ideal sheaf $\iz$ has a resolution given by the Koszul complex induced by $\omega$:
$$
0 \longrightarrow \bigwedge^3 \left(\tp3(-d-2) \right)\longrightarrow \bigwedge^2 \left(\tp3(-d-2) \right)\longrightarrow \tp3(-d-2) \longrightarrow \iz \longrightarrow 0 .
$$
We break it into short exact sequences that, with the appropriate identifications, are:
\begin{align}
0 \longrightarrow\mathcal{O}_{\mathbb{P}^3}(-3d-2)\longrightarrow \, & \Omega^1_{\mathbb{P}^3}(-2d) \longrightarrow T_\sF (-d-2) \longrightarrow 0, \label{1steq} \\
0 \longrightarrow T_\sF (-d-2)\longrightarrow \, & \tp3(-d-2) \longrightarrow \iz \longrightarrow 0. \label{2ndeq}
\end{align}
From \eqref{2ndeq} and the Euler sequence we note that the image of $H^0(\tp3(-d-2+l)) \rightarrow H^0(\mathcal{I}(l))$ is precisely $I_l$. On the other hand \eqref{1steq} and the Bott formulae imply 
$$
h^1(T_\sF (-d-2+l)) = \begin{cases} 0 & l \neq 2d \\ 1 & l= 2d
\end{cases}.
$$ 
Therefore $I_l = I(Z)_l$ except for $l=2d$ where the latter has one more generator. We only need to produce this missing generator. Since $I(Z)$ is the saturation of the ideal generated by $A_0, \dots, A_3$, we have that $s\in I(Z)$ if there exist $m\geq 1$ and a $4 \times 4$ matrix $Q$ with coefficients in $S$ such that 
$$
s\begin{bmatrix}
x_0^m \\ \vdots\\ x_3^m
\end{bmatrix} = Q \begin{bmatrix}
A_0 \\ \vdots\\ A_3
\end{bmatrix}.
$$
Let $B$ be the matrix defined by $d\omega$ so that the Euler relation $\iota_Rd\omega = (d+2)\omega$ translates to 
$$
B \begin{bmatrix}
x_0 \\ \vdots\\ x_3
\end{bmatrix} = (d+2)\begin{bmatrix}
A_0 \\ \vdots\\ A_3
\end{bmatrix}.
$$
To show that $P = \operatorname{Pf}(B) \in I(Z)$ we set $m=1$ and $Q = (q_{ij})$ where $(-1)^{i+j}q_{ij}$ is the Pfaffian of the submatrix of $B$ obtained by eliminating the rows $i,j$ and the columns $i,j$.
\end{proof}

From the previous discussion and the Theorem we derive the following corollary.

\begin{coro}\label{cor:generic}
Let $\sF $ be a generic codimension one distribution of degree $d$ on $\mathbb{P}^n$. Then $\sF $ is determined by its singular scheme if $(d,n) \notin \{(1,2),(1,3)\}$ .
\end{coro}

Let $\dd^{gen}(\pn)$ denote the moduli space of isomorphism classes of generic distributions on $\pn$, and set
$$ \phi(d,n) := \displaystyle\int_{\mathbb{P}^n}c_n(\Omega_{\mathbb{P}^n}^1(d+2))
= \frac{(d+1)^{n+1} - (-1)^{n+1}}{d+2}. $$
The morphism of Theorem \ref{sing:map} then becomes
$$ \Sigma \colon \dd^{gen}(\pn) \longrightarrow {\rm Hilb}^{\phi(d,n)}(\mathbb{P}^n) $$
Corollary \ref{cor:generic} implies that this morphism is injective when $(d,n) \notin \{(1,2),(1,3)\}$. For these two exceptional cases, we have that $\Sigma$ is dominant. Indeed, note that $\phi(1,2)=3$ and $\phi(1,3)=5$. In both cases the reduced schemes in linear general position are all projectively equivalent. Since they are dense in the respective Hilbert schemes, we just need to show one example for each. Then consider the distribution given, in homogeneous coordinates $(x_0: \dots:x_n)$, by 
\[
\omega = \sum_{j=0}^n (x_{j+1}x_j -x_{j-1}^2) dx_j
\]
where the indices are taken modulo $n+1$. 
One may check that the singular scheme is supported on the points $(\lambda^{a_0} : \dots : \lambda^{a_n})$, where $a_j = \frac{1- (-2)^j}{3}$ and $\lambda$ is a root of unity of order $\phi(1,n)$. Hence the singular scheme must be reduced.

On the other hand, Lemma \ref{lem:generic} implies that $\Sigma\colon \dd^{gen}(\pn) \rightarrow {\rm Hilb}^{\phi(d,n)}(\mathbb{P}^n)$ is not surjective, for any $(d,n)$. Precisely determining the image of $\Sigma$ seems to be a very difficult problem. Let $Z$ be a zero-dimensional subscheme of $\mathbb{P}^3$ of length $5$. In order to be the singular scheme of a degree one distribution, it is necessary that $Z$ must be nonplanar, arithmetically Gorenstein and locally complete intersection, see the proof of Theorem \ref{resolution} and \cite[p.119]{SCH}. We wonder if these conditions are also sufficient. 

This concludes our discussion for the generic distributions - corresponding to the first line of Table \ref{table deg 1 v2}. In the next section we analyze the other three cases.


\section{Distributions of Degree 1}\label{sec:deg1}

In this section, we apply Theorem \ref{sing:map} and Proposition \ref{psyz} to provide a description of the irreducible components of the moduli space of codimension one distributions of degree $1$ on $\p3$. In this framework, we can explicitly describe the image of the morphism $\dd^{P}(\mathbb{P}^3)\rightarrow{\rm Hilb}^{Q}(\mathbb{P}^3)$ and compute its fibers. 

In order to proceed with our intentions, let us recall the classification of codimension one distributions $\ff$ of degree 1 on $\p3$ obtained in \cite[Section 7]{CCJ}. Such distributions are given by exact sequences of the following form
\begin{equation}\label{distdeg1}
0\longrightarrow T_\sF \longrightarrow \tp3 \longrightarrow \iz(3) \longrightarrow 0 \end{equation}
and the classification was given in terms of the Chern classes of the tangent sheaf $T_\ff$. It was shown that $T_\ff$ either splits as a sum of line bundles or is $\mu$-stable; the relevant information is summarized in Table \ref{table deg 1 v2}.



\subsection{Twisted Cubic} 

Let $\dd(0,0)$ be the scheme parameterizing isomorphism classes of codimension one distributions $\ff$ of degree 1 on $\pp^3$ such that $(c_2(T_\sF),c_3(T_\sF))=(0,0)$; according to the classification outlined in Table \ref{table deg 1 v2}, we have that $T_\ff=\op3\oplus\op3(1)$. It follows from \cite[Corollary 8.9]{CCJ} that $C:={\rm Sing}(\ff)$ is an arithmetically Cohen-Macaulay curve (aCM, for short) of degree 3 and genus 0, that is, $C$ is a, possibly degenerate, twisted cubic.

Let $H$ be the irreducible component of ${\rm Hilb}^{3t+1}(\mathbb{P}^3)$ that contains twisted cubics. There exists a map $f \colon H \rightarrow \mathbb{G}(3, H^0(\mathbb{P}^3, \mathcal{O}_{\mathbb{P}^3}(2)))$ sending a curve to the net of quadrics that defines it. The image is the space of determinantal nets and $f$ is an isomorphism on the aCM open subset, see \cite{EPS:Net,IX:twc}. On the other hand, one can easily list the possible irreducible components of an aCM degree three curve. The nonreduced ones have support consisting on one line or the union of two lines; such multiple structures are described in \cite{No:deg3}. It follows that aCM degree three curves defined by a determinantal net of quadrics are projectively equivalent to one of the following:

\begin{enumerate}
	\item $I = (xz-y^2, xw-yz, yw-z^2)$ a twisted cubic;
	\item $I = (xz-y^2,xw,yw)$ a conic meeting a noncoplanar line;
	\item $I = (xw,xy,yz)$ three noncoplanar lines meeting twice;
	\item $I = (xy,xz,yz)$ three noncoplanar concurrent lines;
	\item $I = (x^2,xz, yz)$ a planar double line meeting a line not in its plane;
	\item $I = (x^2,xw, xz-yw)$ a double line of genus -1 meeting another line;
	\item $I = (x^2-yz, xz,z^2)$ a triple line in a quadratic cone;
	\item $I = (x^2,xy,y^2)$ a infinitesimal neighborhood of a line.
\end{enumerate}

Next, we use the results in the previous section to describe $\dd(0,0)$ in terms of the morphism $\Sigma\colon \dd(0,0)\rightarrow H$ and compute its dimension. For each ideal $I$ listed above, we compute a generic $1$-form vanishing at $V(I)$, showing that the image of the morphism $\Sigma\colon\dd(0,0)\rightarrow H$ is precisely the open subset $U\subset H$ consisting of non planar aCM curves apart from triple lines in case (viii). We then find that the space of linear first syzygies of the ideals in the list above have, in all cases, dimension equal to 2, thus, according to Proposition \ref{fibra} and Theorem \ref{psyz}, the fibers of $\Sigma$ are open subsets of $\mathbb{P}^1$. We infer that they are of the form $D_+(h)$ for some homogeneous polynomial $h\in\com[t_0,t_1]$. As a consequence of the theorem on the dimension of the fibers, we conclude that $\dd(0,0)$ is irreducible and
$$ \dim \dd(0,0) = \dim H + 1 = 13. $$
Moreover, one can apply Proposition \ref{smooth} to check that $\dd(0,0)$ is smooth, since 
$$ \Ext^1(T_\ff,\iz(3)) = H^1(\iz(3)) \oplus H^1(\iz(2)) = 0. $$

Using the procedure outlined in Remark \ref{algo}, we can find explicit expressions for 1-forms vanishing at each of the possible ideals listed above. In addition, since the tangent sheaf is $\op3\oplus\op3(1)$, there exists a constant vector field that is tangent to each distribution. All of this information (generic 1-form vanishing at $V(I)$, vector field and polynomial $h$) is listed in Table \ref{twistedcubic}.

\begin{table}[ht]
\makegapedcells
\begin{tabular}{|c | c | c | c |} \hline
 Case & Generic $1$-form & Vector field & $h$ \\ \hline\hline
(i) & \makecell[l]{$\omega = t_1(z^2-yw)dx + (t_0(z^2-yw) - t_1(yz-xw))dy + $\\ $\quad + (t_1(y^2-xz) -t_0(yz-xw))dz + t_0(y^2-xz)dw$ } & \makecell[l]{$ t_0^3\partial_x + t_0^2t_1 \partial_y +$ \\$ \quad + t_0t_1^2 \partial_z+ t_1^3 \partial_w$ } &1 \\ \hline
(ii) & \makecell[l]{$\omega = -t_0ywdx + (t_1yw+ t_0xw)dy + t_1xwdz $\\ $\quad - t_1(y^2-xz)dw$ } & \makecell[l]{$ t_1^2\partial_x + t_0t_1 \partial y + $ \\$ \quad + t_0^2 \partial z$ } & $t_1$ \\ \hline
(iii) & \makecell[l]{$\omega = t_1yzdx - t_0xwdy - t_1xydz +t_0xydw$ } & \makecell[l]{$ t_0\partial_z + t_1 \partial_w $ } & $t_0t_1$ \\ \hline
(iv) & \makecell[l]{$\omega = t_0yzdx + t_1xzdy - (t_0+t_1)xydz$ } & \makecell[l]{$ \partial_w $ } & $t_0t_1(t_0+t_1)$\\ \hline
(v) & \makecell[l]{$\omega = (t_0yz+ t_1xz)dx - t_0xzdy - t_1x^2dz$ } & \makecell[l]{$ \partial_w $ } & $t_0t_1$ \\ \hline
(vi) & \makecell[l]{$\omega = (t_0xw+ t_1(yw-xz))dx - t_1xwdy + t_1x^2dz +$\\ $ \quad -t_0x^2dw$ } & \makecell[l]{$ t_0\partial_z + t_1 \partial_w$} & $t_1$ \\ \hline
(vii) & \makecell[l]{$\omega = (t_0z^2+ t_1xz)dx - t_1z^2dy - (t_0xz +t_1(x^2-yz))dz$ } & \makecell[l]{$ \partial_w $ } & $t_1$ \\ \hline
\end{tabular}
\medskip
\caption{Classification of distributions with singular scheme in ${\rm Hilb}^{3t+1}(\mathbb{P}^3)$.}
 \label{twistedcubic}
\end{table}

The case (viii) does not figure in Table \ref{twistedcubic} because a $1$-form singular at such scheme must be singular in codimension one. Indeed, the generic $1$-form is $\omega = (t_0y+t_1x)(ydx-xdy)$. We also note that the integrable distributions, i.e. foliations, that arise in this case must be of linear pullback type $\mathcal{LPB}(1)$; this follows from Jouanolou's classification \cite{Jou}.  In particular, the singular scheme must be a cone over $3$ points in a plane, these are precisely the cases (iv), (v) and (vii).
Our conclusions can therefore be summarized in the following statement.

\begin{teo}\label{cubic}
Let $\dd(0,0)$ be the scheme parameterizing distributions (up to isomorphism) on $\pp^3$ of codimension one and degree $1$ whose tangent sheaf is locally free. Let also $H$ be the irreducible component of ${\rm Hilb}^{3t+1}(\pp^3)$ containing twisted cubics. Then:
\begin{itemize}
\item[(i)] There exists a surjective morphism $\Sigma \colon\dd(0,0)\rightarrow U$, where $U\subset H$ is the open set consisting of non planar aCM curves in $H$ which are not of type (viii), which assigns each $\ff\in\dd(0,0)$ to its singular scheme, and such that each fiber $\Sigma^{-1}(C)$ is naturally isomorphic to an open set of $\pum$. 
\item[(ii)] $\dd(0,0)$ is an irreducible, smooth quasi-projective variety of dimension 13.
\item[(iii)] There exists a map $v\colon\dd(0,0)\rightarrow \mathbb{P}^3$ that assigns to each $\ff\in\dd(0.0)$ a constant vector field $v(\ff)\in\pp^3$ which is tangent to $\ff$. Moreover, for any $C\in U$, the closure of the image $\overline{v(\Sigma^{-1}(C))}$ describes a rational normal curve of degree $d$ in some $\pp^d\subset\pp^3$, where: 
\begin{itemize}
\item[(a)] $d=3$ if and only if $C$ is a twisted cubic.
\item[(b)] $d=2$ if and only if $C$ is a conic meeting a noncoplanar line;
\item[(c)] $d=1$ if and only if $C$ is three noncoplanar lines meeting twice;
\item[(d)] $d=0$ if and only if $C$ is three noncoplanar concurrent lines; 
\end{itemize}
\item[(iv)] $\ff$ is integrable if and only if ${\rm Sing}(\ff)$ falls in cases (iv), (v) or (vi). In this case $\ff$ is a linear pullback foliation. 
\end{itemize}
\end{teo}


\subsection{Line and Two Points}
The Hilbert scheme ${\rm Hilb}^{t+3}(\mathbb{P}^3)$ of configurations of a line and two points is an irreducible variety of dimension $10$, see \cite[Example 4.5]{CN:Det}. The possible configurations arise 
by simply adding two points to a line $X$. The resulting scheme is always a flat limit of $X$ plus two isolated points, \cite[Theorem 1.4]{CN:Det}. Since we are concerned with nonplanar schemes $V(I)\in {\rm Hilb}^{t+3}(\mathbb{P}^3)$, we may consider, up to projective equivalence, the ideals $I$ listed below:

\begin{enumerate}
 \item $I = (x,y) \cap (x,z,w) \cap (y,z,w) = ( yw, xw, yz, xz ,xy)$, a line and two points;
 \item $I = (x,y) \cap (x,y^2,z) \cap (y,z,w) = ( xw, yz, xz, y^2, xy)$, a line and two points but one embedded;
 \item $I = (x,y) \cap (x,y^2,z) \cap (x^2, y,w) = ( xw, yz, x^2 , y^2, xy)$, a line with two simple embedded points;
 \item $I =(x,y) \cap (x^2,z,w)= (yw, xw, yz, xz, x^2)$, a line and a double point;
 \item $I = (x,y)\cdot(x,y,z) = (x^2 , xy, xz, y^2 , yz) $, a line with embedded double point.
\end{enumerate}

The corresponding ideal sheaves all have the same type of resolution:
\begin{equation}\label{res line 2pt}
0 \longrightarrow \mathcal{O}_{\mathbb{P}^3}(-4)^{\oplus 2} \longrightarrow \mathcal{O}_{\mathbb{P}^3}(-3)^{\oplus 6} \longrightarrow \mathcal{O}_{\mathbb{P}^3}(-2)^{\oplus 5} \longrightarrow \mathcal{I}_Z \longrightarrow 0
\end{equation}
In particular, one can check that $h^i(\iz(k))=0$ for $i>0$ and $k>0$.

Now we analyze each of the cases raised above following the same procedure outlined in the previous subsection. Generic 1-forms vanishing at $V(I)$ for the ideals in items (i) through (v) are obtained following the algorithm given in Remark \ref{algo}; however, every 1-form vanishing at a scheme $Z=V(I)$ for an ideal $I$ of type (v) also vanishes along three concurrent lines meeting at the embedded double point, implying that the corresponding morphism $\tp3\rightarrow\iz(3)$ cannot be surjective; the explicit expression is in display \eqref{type5} below. It thus follows that the image of the morphism $\Sigma\colon\dd(1,1)\rightarrow{\rm Hilb}^{t+3}(\mathbb{P}^3)$ is the open subset $U\subset{\rm Hilb}^{t+3}(\mathbb{P}^3)$ consisting of nonplanar schemes without embedded double point (that is, the complement of the families of schemes of type (v) and the planar ones). 

Furthermore, we also compute the dimension of the fibres of $\Sigma$ using Proposition \ref{psyz}; it turns out that the $h^0(\Omega^1_{\p3}\otimes\iz(3))=6$ for each possible scheme $Z$ from the list above. Finally, the open subsets giving the fibers of the morphism $\Sigma$ are of the form $D_+(h)$, for some homogeneous polynomial $h\in\com[t_0,\dots,t_5]$. We list these polynomials in Table \ref{table:line+2pt}, together with the generic 1-form of each case.

\begin{table}[ht] 
\makegapedcells
\begin{tabular}{|c | c |c| } \hline
 Case & Generic $1$-form & $h$ \\ \hline\hline
(i) & \makecell[l]{ $\omega = (t_2yw+t_4yz)dx + (t_3xw+t_5xz)dy +$ \\ $\quad \,+(t_0yw+t_1xw-(t_4+t_5)xy)dz+$ \\$ \quad \,-(t_0yz+t_1xz + (t_2+t_3)xy)dw $ } & $t_0t_1(t_3t_4-t_2t_5)$ \\ \hline
(ii) & \makecell[l]{ $\omega = (t_3yz-t_5y^2)dx + (t_1yz +t_2xw +t_4xz+t_5xy)dy +$ \\ $\quad \,+(t_0xw-t_1y^2-(t_3+t_4)xy)dz -(t_0xz + t_2xy)dw $ } & $t_0t_3(t_0t_5-t_2t_3)$ \\ \hline
(iii) & \makecell[l]{ $\omega = (t_2yz+t_3y^2+t_4xw+t_5xy)dx +$ \\ $\quad \,+ (t_0yz+t_1xw-t_3xy-t_5x^2)dy +$ \\ $\quad \,-(t_0y^2+t_2xy)dz -(t_1xy + t_4x^2)dw $ } & $t_0t_4(t_1t_2-t_0t_4)$ \\ \hline
(iv) & \makecell[l]{ $\omega = (t_0 yw +t_3yz + t_4xw +t_5xz)dx - (t_0xw + t_3xz)dy +$ \\ $\quad \,+(t_1yw+t_2xw-t_5x^2)dz -(t_1yz+ t_2xz+t_4x^2)dw $ } & $t_1(t_4t_3-t_0t_5)$ \\ \hline
\end{tabular}
\caption{Classification of distributions with singular scheme in ${\rm Hilb}^{t+3}(\mathbb{P}^3)$.}
 \label{table:line+2pt}
\end{table}

We remark that the generic $1$-form in case (v) is 
\begin{equation}\label{type5}
\omega = (t_1xz + t_3y^2 +t_4xz+t_5xy)dx +(t_0yz+t_2xz)dy-(t_0y^2 + (t_1+t_2)xy+t_4x^2)dz.
\end{equation}
Since $\omega$ does not depend on $w$, it defines a family of linear pullback foliations. Each general member is singular at three lines crossing at the double point.

\begin{teo}\label{line+2pt}
The scheme $\dd(2,2)$ parameterizing isomorphism classes of codimension one distributions $\ff$ of degree 1 on $\pp^3$ such that $(c_2(T_\ff),c_3(T_\ff))=(2,2)$, is an irreducible smooth quasi-projective variety of dimension 15. The image of the morphism $\Sigma\colon\dd(2,2)\rightarrow{\rm Hilb}^{t+3}(\mathbb{P}^3)$ is the open subset consisting of nonplanar schemes without an embedded double point, and its fibers are open subsets of $\pp^5$.
\end{teo}

The second part of the previous theorem was proved above, and it implies, by the theorem of the dimension of the fibres, that $\dd(2,2)$ is an irreducible quasi-projective variety of dimension 15. In order to show that $\dd(2,2)$ is smooth, we argue, using Lemmata \ref{tec-lema} and \ref{lem:dimext}, that $\dim T_{[\sF]}\dd(2,2)=15$ for every $\ff\in\dd(2,2)$. Indeed, note that $T_\sF^{\vee}$ is a rank 2 stable reflexive sheaf with Chern classes $(c_1(T_\sF^{\vee}),c_2(T_\sF^{\vee}),c_3(T_\sF^{\vee}))=(-1,2,2)$ so we can see from \cite[Table 2.6.1]{Chang} that 
\[
h^1(T_\sF^{\vee}(k))=0 \text{ for } k \neq 0, \quad h^2(T_\sF^{\vee}(k))=0 \text{ for } k\geq -1 \, \text{ and } \, h^3(T_\sF^{\vee}(k))=0 \text{ for } k\geq -4.
\]
We can then apply the functor $\ponto\otimes T_\sF^{\vee}(3)$ to the resolution in display \eqref{res line 2pt}, keeping in mind item (2) of Lemma \ref{tec-lema}, to conclude that $h^i(T_\sF^{\vee}\otimes\iz(3))=0$ for $i>0$, thus $\Ext^i(T_\sF,\iz(3))=0$ for $i>1$ and
$$ \Ext^1(T_\sF,\iz(3)) \simeq H^0(\inext^1(T_\sF,\op3)\otimes\iz(3)) \simeq H^0(\inext^3(\mathcal{O}_Z,\op3)\otimes\iz), $$
with the last isomorphism coming from the sequence that defines the distribution, namely \eqref{distdeg1}. The expression in display \eqref{dim-hom} gives, in the case at hand,
$$ \dim\Hom(T_\sF,\iz(3)) = \chi(T_\sF,\iz(3)) + h^0(\inext^3(\mathcal{O}_Z,\op3)\otimes\iz) = 9+6 = 15, $$
with the two quantities in the middle being computed using \eqref{chi(ef)} and Lemma \ref{lem:dimext}. 

To compute $ \chi(T_\sF,\iz(3)) $ we need to compute ${\rm ch}(T_\sF)^\vee$ and ${\rm ch}(\iz(3))$. Since we have $c_1(T_\sF) = 1$, $c_2(T_\sF) = c_3(T_\sF) = 2$, $c_1(\iz(3)) = 3$, $c_2(\iz(3)) =1$ and $c_3(\iz(3)) = -5$,
it follows that ${\rm ch}(T_\sF)^\vee = \left(2,-1, -\frac{3}{2}, -\frac{1}{6} \right)$ and ${\rm ch}(\iz(3)) = \left(1,3,\frac{7}{2}, \frac{1}{2}\right)$. Therefore $\chi(T_\sF,\iz(3)) = 9$.

Finally we compute $h^0(\inext^3(\mathcal{O}_Z,\op3)\otimes\iz)= {\rm length }\, \inext^3(\mathcal{O}_Z,\op3)\otimes\iz $. We know that $Z$ is composed by a line $L$ and two points $p_1$ and $p_2$. Table \ref{table:line+2pt} shows that either: $p_1 , p_2 \not \in L$ are isolated points; $p_1 \notin L $ and $p_2 \in L$ is a simple embedded point; $p_1, p_2 \in L$ are simple embedded points; or $p_1=p_2 \notin L$ is a double point. In any case, Lemma \ref{lem:dimext} yields $h^0(\inext^3(\mathcal{O}_Z,\op3)\otimes\iz) = 6$.


\subsection{Conic and One Point}
Recall that the Hilbert Scheme ${\rm Hilb}^{2t+2}(\mathbb{P}^3)$ consists of two irreducible components: one component contains two skew lines, while the other contains non pure schemes consisting of a conic plus a point; let $H$ denote this second component. One can show that $H$ is a smooth irreducible variety of dimension $11$, see \cite[Example 4.6(b)]{CN:Det}. We will only consider the configurations where the point (embedded or not) does not lie on the same plane as the conic. Then, up to projective equivalence, they are listed below:
\begin{enumerate}
 \item $I = (xy-z^2,w)\cap (x,y,z)= (zw, yw, xw, xy - z^2 )$, a conic and a point;
 \item $I = (xy,w)\cap (x,y,z) = (zw, yw, xw, xy)$, a pair of concurrent lines and a point;
 \item $I = (z^2,w)\cap (x,y,z) = (zw, yw, xw, z^2 )$, a planar double line and a point;
 \item $I = (xy-z^2,w)\cap (y,z,w^2) = (w^2 , zw, yw, xy - z^2 )$, a conic and an embedded point;
 \item $I = (xy,w)\cap (x,z, w^2) = (w^2 , zw, xw, xy)$, a pair of concurrent lines and an embedded point in the smooth locus;
 \item $I = (xy,w)\cap (x,y,w^2) = (w^2 , yw, xw, xy)$, a pair of concurrent lines and an embedded point in the singular locus;
 \item $I = (z^2,w)\cap (y,z,w^2) = (w^2 , zw, yw, z^2 )$, a planar double line and an embedded point.
\end{enumerate}

The corresponding ideal sheaves have all the same resolution:
\begin{equation}\label{res conic pt}
0 \longrightarrow \mathcal{O}_{\mathbb{P}^3}(-4) \longrightarrow \mathcal{O}_{\mathbb{P}^3}(-3)^{\oplus 4} \longrightarrow \mathcal{O}_{\mathbb{P}^3}(-2)^{\oplus 4} \longrightarrow \mathcal{I}_Z \longrightarrow 0
\end{equation}
Again, one can check that $h^i(\iz(k))=0$ for $i>0$ and $k>0$.

Now we analyze each of the cases listed above following the procedure outlined in the previous subsections. Generic 1-forms vanishing at $V(I)$ for the ideals in items (i) through (vii) are obtained using the algorithm described in Remark \ref{algo}; we then find that every 1-form vanishing at a scheme $Z=V(I)$ for ideals $I$ of type (vi) or (vii) also vanishes along three concurrent lines, meaning that the corresponding morphism $\tp3\rightarrow\iz(3)$ cannot be surjective; the explicit expressions are in display \eqref{type67} below. It thus follows that the image of the morphism $\Sigma\colon\dd(2,2)\rightarrow H$ is the open subset $U\subset H$ consisting of nonplanar schemes which are not of type (vi) or (vii).

Finally, one can check that the dimension of the space of linear syzygies of the ideals of type (i) through (v) is equal to 4, thus the fibers of $\Sigma\colon\dd(2,2)\rightarrow H$ are open subsets of $\p3$ the form $D_+(h)$ for some homogeneous polynomial $h\in\com[t_0,t_1,t_2,t_3]$. We provide explicit expressions for the generic $1$-form and corresponding polynomials $h$ in Table \ref{table:conic+pt}.

\begin{table}[ht]
\makegapedcells
\begin{tabular}{|c | c | c| } \hline
 Case & Generic $1$-form & $h$ \\ \hline\hline
(i) & \makecell[l]{ $\omega = w(t_1z+t_2y)dx + w(t_0z+(t_3-t_2)x)dy+ $ \\ $\quad \, -w(t_0y+t_1x+t_3z)dz +t_3(z^2-xy)dw $ } & $t_3(t_0t_1+t_2^2-t_2t_3)$ \\ \hline
(ii) & \makecell[l]{ $\omega = w(t_1z+t_2y)dx + w(t_0z+t_3x)dy+ $ \\ $\quad \,-w(t_0y+t_1x)dz -(t_2+t_3)xy dw $ } & $t_0t_1(t_2+t_3)$ \\ \hline
(iii) & \makecell[l]{ $\omega = w(t_1z+t_2y)dx + w(t_0z-t_2x)dy+ $ \\ $\quad \, -w(t_0y+t_1x-t_3z)dz -t_3z^2dw $ } & $t_2t_3$ \\ \hline
(iv) & \makecell[l]{ $\omega = t_3wydx + w(t_1w+t_2z)dy+w(t_0w -t_2y-t_3z)dz+ $ \\ $\quad \, -(t_0zw+t_1yw+t_3(xy-z^2))dw $ } 
& $t_3(t_0t_2+t_1t_3)$ \\ \hline
(v) & \makecell[l]{ $\omega = w(t_1w+t_2z)dx + t_3xwdy+ $ \\ $\quad \, +w(t_0w-t_2x)dz -(t_0zw+t_1xw+t_3xy)dw $ } 
& $t_0t_2t_3$ \\ \hline
\end{tabular}
\caption{Classification of distributions with singular scheme in ${\rm Hilb}^{2t+2}(\mathbb{P}^3)$.}
 \label{table:conic+pt}
\end{table}

Note that cases (vi) and (vii) do not figure in the table. The generic $1$-forms are, respectively:
\begin{equation}\label{type67}
\begin{aligned}
\omega_6 &= w(t_1w+t_2y)dx + w(t_0w+t_3x)dy-(t_0yw+t_1xw+(t_2+t_3)xy)dw, \\
\omega_7 &= w(t_1w+t_2z)dy + w(t_0w-t_2y+t_3z)dz-(t_0zw+t_1yw+t_3z^2)dw.
\end{aligned}
\end{equation}
They depend on three variables, hence they define families of linear pullback foliations. Each general member of either family is singular at three concurrent lines, the third one passing through the embedded point. 

It is worth pointing out the integrable distributions, i.e. foliations, in this case. These are the rational foliations $\mathcal{R}(2,1)$ singular along a smooth conic and a point. We compute the Frobenius integrability condition $\omega \wedge d\omega =0$ in case (i) and it follows that the integrable distributions are those satisfying $t_0 = t_1 = t_3-2t_2=0$. 

\begin{teo}\label{conic+pt}
Let $\dd(1,1)$ be the scheme parameterizing isomorphism classes of codimension one distributions $\ff$ of degree 1 on $\pp^3$ such that $(c_2(T_\ff),c_3(T_\ff))=(1,1)$. Then the image of the morphism $\Sigma\colon\dd(1,1)\rightarrow H$ is the open subset consisting of non planar schemes in $H$ which are neither of type (vi) nor of type (vii) in the list above, and its fibers are open subsets of $\pp^3$. In addition:
\begin{enumerate}
\item $\dd(1,1)$ is an irreducible smooth quasi-projective variety of dimension 14.
\item $\dd(1,1)$ contains $\mathcal{R}(2,1)$ as a codimension three subset; the image of the restricted morphism $\Sigma\colon\mathcal{R}(2,1)\rightarrow H$ is the open subset consisting of the schemes of type (i) in the list above, and its fibres are open subsets of $\p1$.
\end{enumerate}
\end{teo}

It was already shown in \cite[Proposition 11.2]{CCJ} that $\dd(1,1)$ is smooth but here we give a different proof. Our strategy is to show that $\dim T_{[\sF]}\dd(1,1)=\dim\dd(1,1)$ for every $\ff\in\dd(1,1)$. As it was pointed out in \cite[proof of Theorem 8.5]{CCJ}, the tangent sheaf $T_\sF$ admits the following resolution
\begin{equation}\label{sqc d11}
0 \longrightarrow \op3(-1) \stackrel{\mu}{\longrightarrow} \op3^{\oplus3} \longrightarrow T_\sF \longrightarrow 0. 
\end{equation}
Applying the functor $\Hom(\ponto,\iz(3))$ to this sequence we get
$$ 0 \longrightarrow \Hom(T_\sF,\iz(3)) \longrightarrow H^0(\iz(3))^{\oplus3} \stackrel{\tilde{\mu}}{\longrightarrow} H^0(\iz(4)) \longrightarrow \Ext^1(T_\sF,\iz(3)) \longrightarrow 0 , $$
since $h^1(\iz(3))=0$. Therefore
$$ \dim \Hom(T_\sF,\iz(3)) = 3\, h^0(\iz(3)) - h^0(\iz(4)) + \dim \Ext^1(T_\sF,\iz(3)). $$
Since $h^0(\iz(3))=12$ and $h^0(\iz(4))=25$ we only need to show that $\dim \Ext^1(T_\sF,\iz(3))= 3$. We will use Lemmata \ref{tec-lema} and \ref{lem:dimext} to achieve that.

Note that $T_\sF^\vee=T_\sF(-1)$, since $c_1(T_\sF)=1$, thus $T_\sF^\vee\otimes\iz(3)=T_\sF\otimes\iz(2)$. Tensoring the exact sequence \eqref{sqc d11} with $\iz(2)$ and passing to cohomology, we get 
$$ H^i(\iz(2))^{\oplus 3} \longrightarrow H^i(T_\sF\otimes\iz(2)) \longrightarrow H^{i+1}(\iz(1)), $$
but we already observed right after the sequence \eqref{res conic pt} that $h^i(\iz(k))=0$ for $i>0$ and $k>0$, then $h^i(T_\sF\otimes\iz(2))=0$ for $i>0$. Item (v) in Lemma \ref{tec-lema} yields
$$ \dim\Ext^1(T_\sF,\iz(3)) = h^0(\inext^1(T_\sF, \op3) \otimes\iz(3)) = h^0(\inext^3(\mathcal{O}_Z, \op3) \otimes \iz) $$
and this is just the length of $\inext^3(\mathcal{O}_Z, \op3) \otimes \iz$ which will be computed with Lemma \ref{lem:dimext}.

We know that $Z$ is composed of a conic $C$ and a point $p$. If $p \not\in C$ then $\mathcal{I}_{Z,p}= \mathfrak{m}_p$ the maximal ideal at $p$. If $p\in C$, i.e., an embedded point for $Z$. From the table \ref{table:conic+pt} we know that we only need to treat the case where $C$ is reduced and $p$ is embedded in the regular part. In either case, Lemma \ref{lem:dimext} shows that the length is three hence $\dim\Ext^1(T_\sF,\iz(3)) = 3$.


\appendix
\section{Tensor-Hom Relation for Sheaves} \label{Apdx}
If $M$ and $N$ are modules over a commutative ring $R$, then there exists a canonical morphism of $R$-modules
\begin{align*}
 M^\vee \otimes N & \longrightarrow \Hom_R(M , N) \\
 u\otimes y & \longmapsto (x \mapsto u(x)\otimes y )
\end{align*} 
see \cite[Chapter II \S 4.2]{Bourbaki}. Now, given sheaves $F$ and $G$ on a variety $X$, the previous observation leads to the following morphism of sheaves
$$ \eta\colon F^\vee\otimes G \longrightarrow \inhom(F,G) $$
which we call \emph{Tensor-Hom relation}. The following technical lemma is useful in our proof of smoothness for $\dd(2,2)$ and $\dd(1,1)$.

\begin{lema}\label{tec-lema}
Let $X$ be a smooth projective threefold and let $F$ and $G$ be sheaves on $X$. If $F$ is reflexive and $G$ is torsion free then:
\begin{enumerate}
\item $\inext^1(F,G) \simeq \inext^1(F,\mathcal{O}_X)\otimes G$;
\item $\Tor_1(F,G)=0$;
\item The morphism $\eta$ is surjective and the induced cohomology map $H^i(F^\vee\otimes G)\rightarrow H^i(\inhom(F,G))$ is surjective for $i=0$ and an isomorphism for $i\ge 1$;
\item $\eta$ is an isomorphism if either $\sing(F)\cap\sing(G)=\emptyset$ or $G$ is reflexive;
\item if $H^i(F^\vee\otimes G)=0$ for $i>0$, then $\Ext^1(F,G)=H^0(\inext^1(F,\mathcal{O}_X)\otimes G)$ and $\Ext^i(F,G)=0$ for $i>1$.
\end{enumerate}
\end{lema}
\begin{proof}
Since $F$ is reflexive, it admits a resolution of the form
\begin{equation} \label{sqc1}
0 \longrightarrow L_1 \stackrel{\mu}{\longrightarrow} L_0 \longrightarrow F \longrightarrow 0 ,
\end{equation}
where $L_1$ and $L_0$ are locally free sheaves. Applying the functor $\ponto \otimes G$ to this sequence we get
$$ 0 \longrightarrow \Tor_1(F,G) \longrightarrow L_1\otimes G \stackrel{\mu\otimes\mathbf{1}_G}{\longrightarrow} L_0\otimes G \longrightarrow F\otimes G \longrightarrow 0 . $$ 
Note that $L_0\otimes G$ is torsion free, while the support of $\Tor_1(F,G)$ is contained in $\sing(F)\cap\sing(G)$. Then $\Tor_1(F,G) \subset L_1\otimes G$ is only possible if $\Tor_1(F,G)=0$. This proves (ii).

Dualizing the sequence in display \eqref{sqc1}, we obtain
\begin{align}\label{sqc1d}
0 \longrightarrow F^\vee \longrightarrow & L_0^\vee \stackrel{\mu^\vee}{\longrightarrow} K \longrightarrow 0, \\
\label{sqc1da}
0 \longrightarrow K \longrightarrow & L_{ 1}^\vee \longrightarrow \inext^1(F,\mathcal{O}_X) \longrightarrow 0,
\end{align}
where $K:=\img\mu^\vee$. On the one hand, we apply the functor $\ponto\otimes G$ to the sequence in display (\ref{sqc1d}) and obtain 
\begin{equation} \label{sqc2:tensor}
 0 \longrightarrow \Tor_1(K,G) \longrightarrow F^\vee\otimes G \longrightarrow L_0^\vee\otimes G \xrightarrow{\mu^\vee\otimes\mathbf{1}_G} K\otimes G \longrightarrow 0
\end{equation}
On the other hand, the apply the functor $\inhom(\ponto,G)$ to the sequence in display \eqref{sqc1} to obtain
\begin{equation} \label{sqc2:hom}
0 \longrightarrow \inhom(F,G) \longrightarrow L_0^\vee\otimes G 
\xrightarrow{\mu^\vee\otimes\mathbf{1}_G} { L_1^\vee\otimes G} \longrightarrow \inext^1(F,G) \longrightarrow 0.
\end{equation}
As tensor product preserves cokernels, we get that $\inext^1(F,G) \simeq \inext^1(F,\mathcal{O}_X) \otimes G$. This proves (i). 

The sequence (\ref{sqc2:hom}) also shows that $\ker \mu^\vee\otimes\mathbf{1}_G = \inhom(F,G) $ and the sequence in display (\ref{sqc2:tensor}) induces
$$ 
0 \longrightarrow \Tor_1(K,G) \longrightarrow F^\vee\otimes G \stackrel{\eta}{\longrightarrow} \inhom(F,G) \longrightarrow 0. 
$$
In the same fashion, we apply the functor $\ponto\otimes G$ to the sequence in display (\ref{sqc1da}) to obtain
\[
\Tor_2(\inext^1(F,\mathcal{O}_X),G) \simeq \Tor_1(K,G);
\]
if $\sing(F)\cap\sing(G)=\emptyset$, then $\Tor_2(\inext^1(F,\mathcal{O}_X),G)=0$ hence $\eta$ is an isomorphism. Alternatively, item (2) also implies that $\Tor_1(K,G)=0$ when $G$ is reflexive, since the sequence (\ref{sqc1d}) implies that $K$ is torsion free. This completes the proof of item (iv).

Moreover, $\Tor_1(K,G)$ is a zero-dimensional sheaf whenever non trivial, so the induced cohomology map
$$ H^i(F^\vee\otimes G) \longrightarrow H^{i}(\inhom(F,G)) $$
is surjective for $i=0$ and an isomorphism for $i\ge 1$. This completes the proof of (iii).

Finally, the item (v) is an immediate consequence of the spectral sequence for local to global Ext's. Indeed, if $h^i(F^\vee\otimes G)=0$ for $i>0$, then $h^i(\inhom(F,G))=0$ in the same range by item (iii). Since $\dim\inext^j(F,G)=0$ due to the fact that $F$ is reflexive, the only nontrivial entries in the second page of the spectral sequence $E_2^{ij}=H^i(\inext^j(F,G))$ occur for $(i,j)=(0,0),(0,1)$, thus, in addition, all differentials $d_2^{ij}$ must vanish. It follows that $E_2^{ij}=H^i(\inext^j(F,G))$ already converges in the second page, therefore $\Ext^1(F,G)=H^0(\inext^1(F,G))$ and $\Ext^k(F,G)=0$ for $k=2,3$.
\end{proof}

Next, we recall the following version of the Hirzebruch--Riemann--Roch theorem, compare with \cite[Lemma 6.1.3]{HL}:
\begin{equation} \label{chi(ef)}
\chi(E,F) := \sum_k (-1)^k \dim\Ext^k(E,F) = \int_X {\rm ch}(E)^\vee\cdot {\rm ch}(F)\cdot{\rm td}(X)
\end{equation}
If we know that $\Ext^k(E,F)=0$ for $k>1$ (situation described in the Lemma \ref{tec-lema}), then
\begin{equation} \label{dim-hom}
\dim\Hom(E,F) = \chi(E,F) + \dim\Ext^1(E,F).
\end{equation}

The calculation of $\dim\Ext^1(E,F)$ in our examples will rely on the following technical fact from commutative algebra.


\begin{lema}\label{lem:dimext}
Let $\kappa$ be a field and let $R = \kappa[x,y,z]_{(x,y,z)}$ be the local ring at the origin of $\mathbb{A}^3_\kappa$. Let $I\subset R$ be an ideal, then
\[
\dim_\kappa \Ext^3_R(R/I, R)\otimes_R I =
\begin{cases}
3, & \text{if $V(I)$ is the closed point;}\\
3, & \text{if $V(I)$ is a smooth rational curve with a simple embedded point;} \\
6, & \text{if $V(I)$ is a double point.}
\end{cases}
\]
\end{lema}

\begin{proof}
First assume that $I = (x,y,z)$ the maximal ideal. Therefore $\Ext^3_R(R/I, R) = R/I$ and 
\[
\Ext^3_R(R/I, R)\otimes_R I = \frac{(x,y,z)}{(x,y,z)^2} \simeq \kappa^3.
\]

Next assume that $V(I)$ is a smooth rational curve with a simple embedded point. Up to automorphism of $R$ we can suppose that $I = (x,y)\cap (x^2,y,z) = (x^2,xz,y)$ which gives the following resolution
\[
0 \longrightarrow R \xrightarrow{\begin{pmatrix}
-y \\ -x \\ z
\end{pmatrix} } R^3 \xrightarrow{
\begin{pmatrix}
z & 0 & y \\ -x & y & 0 \\ 0 & -xz & -x^2 
\end{pmatrix} 
} R^3 \xrightarrow{
\begin{pmatrix}
x^2 & xz & y
\end{pmatrix}
} R \longrightarrow R/I \longrightarrow 0
\]
Therefore $\Ext^3_R(R/I, R)$ is determined by the leftmost map in this sequence, indeed it is the cokernel of its dual, i.e., $\Ext^3_R(R/I, R) = R/(x,y,z)$. It follows that
\[
\Ext^3_R(R/I, R)\otimes_R I = \frac{(x^2,xz,y)}{(x^2,xz,y)\cdot (x,y,z)} = \frac{(x^2,xz,y)}{(x^3,x^2z,x^2z,xy,y^2,yz) } \simeq \kappa^3.
\]

Now consider that $V(I)$ is a double point. We will follow the same path. Up to automorphism, $I = (x,y,z)^2 + (x,y) = (x,y,z^2)$. Then we have the resolution
\[
0 \longrightarrow R \xrightarrow{\begin{pmatrix}
x \\ y \\ z^2
\end{pmatrix} } R^3 \xrightarrow{
\begin{pmatrix}
0 & z^2 & -y \\ -z^2 & 0 & x \\ y & -x & 0 
\end{pmatrix} 
} R^3 \xrightarrow{
\begin{pmatrix}
x & y & z^2
\end{pmatrix}
} R \longrightarrow R/I \longrightarrow 0
\]
and it follows that $\Ext^3_R(R/I, R) = R/I$. Therefore
\[
\Ext^3_R(R/I, R) \otimes_R I = \frac{I}{I^2} = \frac{(x,y,z^2)}{(x^2,xy,xz^2,y^2,yz^2,z^4)} \simeq \kappa^6.
\]
We can give a basis taking the classes of $\{ x,y,z^2,xz,yz,z^3\}$.
\end{proof}


\end{document}